\documentclass[12pt,reqno]{amsart}
\usepackage{amsthm,amsfonts,amssymb,euscript}

\newcommand{\bea}{\begin{eqnarray}}
\newcommand{\eea}{\end{eqnarray}}
\def\beaa{\begin{eqnarray*}}
\def\eeaa{\end{eqnarray*}}
\def\ba{\begin{array}}
\def\ea{\end{array}}
\def\be#1{\begin{equation} \label{#1}}
\def \eeq{\end{equation}}

\def\be{{\beta}}

\def\eps{\epsilon}

\def\al{\alpha}

\def\R{{\mathbb{R}}}

\def\Z{{\mathbb{Z}}}

\def\T{{\mathbb{T}}}

\newtheorem{theorem}{Theorem}[section]
\newtheorem{lemma}[theorem]{Lemma}
\newtheorem{proposition}[theorem]{Proposition}
\newtheorem{corollary}[theorem]{Corollary}
\newtheorem{definition}[theorem]{Definition}

\numberwithin{equation}{section}

\begin{document}

\title{The energy-critical defocusing NLS on $\mathbb{T}^3$}

\author{Alexandru D. Ionescu}
\address{Princeton University}
\email{aionescu@math.princeton.edu}

\author{Benoit Pausader}
\address{Brown University}
\email{benoit.pausader@math.brown.edu}

\thanks{The first author was supported in part by a Packard Fellowship.}

\begin{abstract}
We prove global well-posedness in $H^1(\T^3)$ for the energy-critical defocusing initial-value problem
\begin{equation*}
(i\partial_t+\Delta)u=u|u|^4,\qquad u(0)=\phi.
\end{equation*}
\end{abstract}
\maketitle
\tableofcontents

\section{Introduction}\label{Intro}

Let $\mathbb{T}:=\R/(2\pi\mathbb{Z})$. In this paper we consider the energy-critical defocusing equation
\begin{equation}\label{eq1}
(i\partial_t+\Delta)u=u|u|^4
\end{equation}
in the periodic setting $x\in\mathbb{T}^3$. Suitable solutions on a time interval $I$ of \eqref{eq1} satisfy mass and energy conservation, in the sense that the functions
\begin{equation}\label{conserve}
M(u)(t):=\int_{\T^3}|u(t)|^2\,dx,\qquad E(u)(t):=\frac{1}{2}\int_{\T^3}|\nabla u(t)|^2\,dx+\frac{1}{6}\int_{\T^3}|u(t)|^6\,dx,
\end{equation}
are constant on the interval $I$. Our main theorem concerns global well-posedness in $H^1(\T^3)$ for the initial-value problem associated to the equation \eqref{eq1}.

\begin{theorem}\label{Main1} (Main theorem) If $\phi\in H^1(\T^3)$ then there exists a unique global solution $u\in X^1(\R)$ of the initial-value problem
\begin{equation}\label{eq1.1}
(i\partial_t+\Delta)u=u|u|^4,\qquad u(0)=\phi.
\end{equation}
In addition, the mapping $\phi\to u$ extends to a continuous mapping from $H^1(\T^3)$ to $X^1([-T,T])$ for any $T\in[0,\infty)$, and the quantities $M(u)$ and $E(u)$ defined in \eqref{conserve} are conserved along the flow.
\end{theorem}

The uniqueness spaces $X^1(I)\subseteq C(I:H^1(\T^3))$ in the theorem above are defined precisely by Herr--Tataru--Tzvetkov \cite{HeTaTz} and \cite{HeTaTz2}.

The corresponding result in the Euclidean space $\mathbb{R}^3$ was proved by Colliander--Keel--Staffilani--Takaoka--Tao \cite{CKSTTcrit} (see also \cite{B,Cazenave:book,G,KeMe,Ker,KiVi,RV,V}) and is an important tool in our analysis. Motivated by this result, there has been interest to obtain global existence of the defocusing energy-critical equation in more general manifolds, see \cite{He,HeTaTz,HeTaTz2,IoPa,IoPaSt}.

Theorem \ref{Main1} completes both previous results of the authors and G. Staffilani \cite{IoPa,IoPaSt} about the energy-critical nonlinear Schr\"odinger equation on different manifolds, such as $\mathbb{H}^3$ and $\mathbb{R}\times\mathbb{T}^3$, and previous results of Bourgain \cite{Bo1} who proved the global existence of solutions in the energy-subcritical case. It also extends the recent results in Herr--Tataru--Tzvetkov \cite{HeTaTz} who proved global existence of small energy solutions of \eqref{eq1}.

We also refer to \cite{Bo2,BuGeTz,BuGeTz3,BuGeTz2,CKSTT,GePi,Ha} for other results of global existence and large time behavior of subcritical Schr\"odinger equations on compact manifolds.

In this paper we extend and refine the strategy from \cite{IoPa} to the case when no global dispersion is allowed. The main new ingredients that we need are an extinction result which is here provided by Lemma \ref{Extinction}, and a better study of the error term in the construction of an approximate solution in Lemma \ref{AlmostSol}.

The extinction argument is obtained by decomposing the linear propagator into a component which lives on a time interval during which all wave packets travel a distance $\sim 1$, and another component where the wave packets have had time to exit a given ball, but not to refocus more than $o(1)$ percent of their modes.

The analysis of the interaction between nonlinear profiles and linear solutions which are sufficiently far from saturating the Sobolev inequality is done in section \ref{ProofLastLem}. The qualitative fact we need, see for example \cite{AnMa,Ma2}, is that any limit of the quantum measure associated to $\left(\int_0^1\vert\nabla e^{it\Delta}R^J_k\vert^2dt\right)dx$ is absolutely continuous with respect to $dx$. Hence, we expect the effect of the interaction of $e^{it\Delta}R^J_k$ with a concentrating function $N_k^{1/2}\phi(N_kx, N_k^2t)$, $N_k\to+\infty$ to be negligible as $k\to+\infty$. In our case we need strong convergence and quantitative bounds, which we prove in section \ref{ProofLastLem}. 

On the other hand, this does not rule out interaction with a Scale-1 function and it is difficult to adapt the argument from \cite{IoPa} which relies on some smoothing-effect type estimate. However, we note that in order to obtain global existence, it suffices to rule out concentration of the solution on {\it arbitrarily small} time intervals, during which a Scale-1 evolution did not have time to occur.

The arguments we present here appear quite robust and we expect adaptations of them to hold on more general compact manifolds (as long as one has a critical small-data theory), for example on Zoll manifolds which have been recently been studied in \cite{He}.

The rest of the paper is organized as follows. In Section \ref{prelim}, we introduce our notations and state some previous results. In Section \ref{localwp}, we use previous results of Herr-Tataru-Tzvetkov \cite{HeTaTz} to develop a large-data local well-posedness and stability theory for the equation \eqref{eq1}. In Section \ref{Eucl}, we study the behavior of solutions to the linear and nonlinear equation concentrating to a point in space and time. In Section \ref{profiles}, we recall the profile decomposition from \cite{IoPa} (see also \cite{Ker}) to address the loss of compactness of the Sobolev-Strichartz inequality. In Section \ref{proofthm}, we prove Theorem \ref{Main1}, except for a lemma about approximate solutions which is finally proved in Section \ref{ProofLastLem}.

\section{Preliminaries}\label{prelim}
In this section we summarize our notations and collect several lemmas that are used in the rest of the paper.

We write $A\lesssim B$ to signify that there is a constant $C>0$ such that $A\le C B$. We write $A\simeq B$ when $A\lesssim B\lesssim A$. If the constant $C$ involved has some explicit dependency, we emphasize it by a subscript. Thus $A\lesssim_uB$ means that $A\le C(u)B$ for some constant $C(u)$ depending on $u$.

We write $F(z)=z\vert z\vert^4$ the nonlinearity in \eqref{eq1}. For $p\in\mathbb{N}^n$ a vector, we denote by $\mathfrak{O}_{p_1,\dots ,p_n}(a_1,\dots,a_n)$ a $\vert p\vert$-linear expression which is a product of $p_1$ terms which are either equal to $a_1$ or its complex conjugate $\overline{a}_1$ and similarly for $p_j$, $a_j$, $2\le j\le n$.

We define the Fourier transform on $\mathbb{T}^3$ as follows
\begin{equation*}
\left(\mathcal{F}f\right)(\xi):=\frac{1}{(2\pi)^\frac{3}{2}}\int_{\mathbb{T}^3}f(x)e^{-ix\cdot\xi}dx.
\end{equation*}
We also note the Fourier inversion formula
\begin{equation*}
f(x)=\frac{1}{(2\pi)^\frac{3}{2}}\sum_{\xi\in\mathbb{Z}^3}\left(\mathcal{F}f\right)(\xi)e^{ix\cdot\xi}.
\end{equation*}
We define the Schr\"odinger propagator $e^{it\Delta}$ by
\begin{equation*}
\left(\mathcal{F}e^{it\Delta}f\right)(\xi):=e^{-it|\xi|^2}\left(\mathcal{F}f\right)(\xi).
\end{equation*}

We now define the Littlewood-Paley projections. We fix $\eta^{1}:\mathbb{R}\to[0,1]$ a smooth even function with $\eta^1(y)=1$ if $|y|\leq 1$ and $\eta^1(y)=0$ if $|y|\geq 2$. Let $\eta^3:\mathbb{R}^3\to[0,1]$, $\eta^3(\xi):=\eta^1(\xi_1)^2\eta^1(\xi_2)^2\eta^1(\xi_3)^2$. We define the Littlewood-Paley projectors $P_{\le N}$ and $P_N$ for $N=2^j\ge 1$ a dyadic integer by
\begin{equation*}
\begin{split}
&\mathcal{F}\left(P_{\le N}f\right)\left(\xi\right):=\eta^3(\xi/N)\left(\mathcal{F}f\right)(\xi),\qquad\xi\in \Z^3,\\
&P_1f:=P_{\le 1}f,\qquad P_Nf:=P_{\le N}f-P_{\le N/2}f\quad\hbox{if}\quad N\ge 2.
\end{split}
\end{equation*}
For any $a\in (0,\infty)$ we define
\begin{equation*}
P_{\leq a}:=\sum_{N\in 2^{\Z_+},\,N\leq a}P_N,\qquad P_{> a}:=\sum_{N\in 2^{\Z_+},\,N>a}P_N.
\end{equation*}

{\bf{Function spaces.}} The strong spaces are the same as the one used by Herr-Tataru-Tzvetkov \cite{HeTaTz,HeTaTz2}. Namely
\begin{equation*}
\begin{split}
\Vert u\Vert_{X^s(\mathbb{R})}&:=\left(\sum_{z\in\mathbb{Z}^3}\langle \xi\rangle^{2s}\Vert e^{it|\xi|^2}(\mathcal{F}u(t))(\xi)\Vert_{U^2_t}^2\right)^\frac{1}{2},\\
\Vert u\Vert_{Y^s(\mathbb{R})}&:=\left(\sum_{z\in\mathbb{Z}^3}\langle \xi\rangle^{2s}\Vert e^{it|\xi|^2}(\mathcal{F}u(t))(\xi)\Vert_{V^2_t}^2\right)^\frac{1}{2},
\end{split}
\end{equation*}
where we refer to \cite{HeTaTz,HeTaTz2} for a description of the spaces $U^p$ and $V^p$ and of their properties. Note in particular that
\begin{equation*}
X^1(\R)\hookrightarrow Y^1(\R)\hookrightarrow L^\infty(\mathbb{R},H^1).
\end{equation*}

For intervals $I\subset\mathbb{R}$, we define $X^1(I)$ in the usual way as restriction norms, thus
\begin{equation*}
X^1(I):=\{u\in C(I:H^1):\Vert u\Vert_{X^1(I)}:=\sup_{J\subseteq I,\,|J|\leq 1}[\inf\{\Vert v\Vert_{X^1(\mathbb{R})}:v_{\vert J}=u\}]<\infty\}.
\end{equation*}
The norm controling the inhomogeneous term on an interval $I=(a,b)$ is then defined as
\begin{equation}\label{NNorm}
\Vert h\Vert_{N(I)}:=\Big \|\int_a^te^{i(t-s)\Delta}h(s)ds\Big\|_{X^1(I)}.
\end{equation}

We also need a weaker critical norm
\begin{equation}\label{Alex19}
\begin{split}
&\Vert u\Vert_{Z(I)}:=\sum_{p\in\{p_0,p_1\}}\sup_{J\subseteq I,\vert J\vert\le 1}(\sum_NN^{5-p/2}\Vert P_Nu(t)\Vert_{L^{p}_{x,t}(\mathbb{T}^3\times J)}^{p})^{1/p},\\
&p_0=4+1/10,\qquad p_1=100.
\end{split}
\end{equation}
A consequence of Strichartz estimates from Theorem \ref{Stric} below is that
\begin{equation*}
\Vert u\Vert_{Z(I)}\lesssim \Vert u\Vert_{X^1(I)},
\end{equation*}
thus $Z$ is indeed a weaker norm. The purpose of this norm is that it is fungible and still controls the global evolution, as will be manifest from the local theory in Section \ref{localwp}.

{\bf Definition of solutions}.
Given an interval $I\subseteq\R$, we call $u\in C(I:H^1(\mathbb{T}^3))$ a strong solution of \eqref{eq1} if $u\in X^1(I)$ and $u$ satisfies that for all $t,s\in I$,
\begin{equation*}
u(t)=e^{i(t-s)\Delta}u(s)-i\int_s^te^{i(t-t^\prime)\Delta}\left(u(t^\prime)\vert u(t^\prime)\vert^4\right)dt^\prime.
\end{equation*}

{\bf Dispersive estimates}. We recall the following result from \cite{Bo1}.

\begin{theorem}\label{Stric}
If $p>4$ then
\begin{equation*}
\Vert P_Ne^{it\Delta}f\Vert_{L^p_{x,t}(\mathbb{T}^3\times [-1,1])}\lesssim_p N^{\frac{3}{2}-\frac{5}{p}}\Vert P_Nf\Vert_{L^2(\mathbb{T}^3)}.
\end{equation*}
\end{theorem}

As a consequence of the properties of the $U^p_\Delta$ spaces, we have:

\begin{corollary}\label{alex1}
If $p>4$ then for any cube $C$ of size $N$ and any interval $I$, $|I|\leq 1$,
\begin{equation}\label{UpEst}
\Vert P_Cu\Vert_{L^p_{x,t}(\mathbb{T}^3\times I)}\lesssim N^{\frac{3}{2}-\frac{5}{p}}\Vert u\Vert_{U^p_\Delta(I,L^2)}.
\end{equation}
\end{corollary}

We will also use the following results from Herr--Tataru--Tzvetkov \cite{HeTaTz}.

\begin{proposition}[\cite{HeTaTz}, Proposition 2.11]\label{Alex3}
If $f\in L^1_t(I,H^1(\mathbb{T}^3))$ then
\begin{equation}\label{EstimNnorm}
\Vert f\Vert_{N(I)}\lesssim \sup_{\{\Vert v\Vert_{Y^{-1}(I)}\le1\}}\int_{\mathbb{T}^3\times I}f(x,t)\overline{v(x,t)}dxdt.
\end{equation}
In particular, there holds for any smooth function $g$ that
\begin{equation}\label{EstimX1Norm}
\Vert g\Vert_{X^1([0,1])}
\lesssim \Vert g(0)\Vert_{H^1}+(\sum_N\Vert P_N\left(i\partial_t+\Delta\right)g\Vert_{L^1_t([0,1],H^1)}^2)^\frac{1}{2}.
\end{equation}
\end{proposition}


\section{Local well-posedness and stability theory}\label{localwp}

In this section we present large-data local well-posedness and stability results that allow us to connect nearby intervals of nonlinear evolution. A consequence of \cite{HeTaTz} is that the Cauchy problem for \eqref{eq1} is locally well-posed. However, here we want slightly more precise results.

We start with a nonlinear estimate. The goal here is to obtain estimates which are linear in a norm controlling $L^\infty_tH^1$. For this we introduce
\begin{equation}\label{Zprime}
\Vert u\Vert_{Z^\prime(I)}=\Vert u\Vert_{Z(I)}^\frac{1}{2}\Vert u\Vert_{X^1(I)}^\frac{1}{2}.
\end{equation}
We have the following result:

\begin{lemma}\label{alex2}
There exists $\delta>0$ such that if $u_1,u_2,u_3$ satisfy $P_{N_i}u_i=u_i$ with $N_1\ge N_2\ge N_3$ and $|I|\leq 1$, then
\begin{equation}\label{NLEstTor}
\Vert u_1u_2u_3\Vert_{L^2_{x,t}(\mathbb{T}^3\times I)}\lesssim \left(\frac{N_3}{N_1}+\frac{1}{N_2}\right)^\delta\Vert u_1\Vert_{Y^0(I)}\Vert u_2\Vert_{Z^\prime(I)}\Vert u_3\Vert_{Z^\prime(I)}
\end{equation}
and, with $p_0=4+1/10$ as in \eqref{Alex19},  
\begin{equation}\label{Al1}
\Vert u_1u_2u_3\Vert_{L^2_{x,t}(\mathbb{T}^3\times I)}\lesssim N_1^{1/2-5/p_0}N_2^{1/2-5/p_0}N_3^{10/p_0-2}\Vert u_1\Vert_{Z(I)}\Vert u_2\Vert_{Z(I)}\Vert u_3\Vert_{Z(I)}.
\end{equation}
\end{lemma}

\begin{proof}[Proof of Lemma \ref{alex2}] Inequality \eqref{NLEstTor} follows from interpolation beween the following estimate
\begin{equation*}
\Vert u_1u_2u_3\Vert_{L^2_{x,t}}\lesssim N_2N_3\left(\frac{N_3}{N_1}+\frac{1}{N_2}\right)^{\delta_0}\Vert u_1\Vert_{Y^0}\Vert u_2\Vert_{Y^0}\Vert u_3\Vert_{Y^0}
\end{equation*}
from \cite[Proposition 3.5]{HeTaTz}
and the estimate
\begin{equation*}
\Vert u_1u_2u_3\Vert_{L^2_{x,t}}\lesssim \Vert u_1\Vert_{Y^0}\Vert u_2\Vert_{Z}\Vert u_3\Vert_{Z}.
\end{equation*}
To prove this second estimate we observe that if $\{C_k\}_{k\in\mathbb{Z}}$ is a partition of $\Z^3$ in cubes of size $N_2$ then the functions $\left(P_{C_k}u_1\right)u_2u_3$ are almost orthogonal in $L^2_{x}$. Using \eqref{UpEst},
\begin{equation*}
\begin{split}
\Vert u_1u_2u_3\Vert_{L^2_{x,t}(\mathbb{T}^3\times I)}^2&\lesssim\sum_k\Vert \left(P_{C_k}u_1\right)u_2u_3\Vert_{L^2_{x,t}}^2\\
&\lesssim \sum_k\Vert P_{C_k}u_1\Vert_{L^{p_0}_{x,t}}^2\Vert u_2\Vert_{L^{p_0}_{x,t}}^2\Vert u_3\Vert_{L^{2p_0/(p_0-4)}_{x,t}}^2\\
&\lesssim \Vert u_2\Vert_{Z(I)}^2\Vert u_3\Vert_{Z(I)}^2\sum_k\Vert P_{C_k}u_1\Vert_{U^{p_0}_\Delta(I,L^2)}^2.
\end{split}
\end{equation*}
Using that $Y^0(I)\hookrightarrow U^{p_0}_\Delta(I,L^2)$ and remarking that the $Y^0$-norm is square-summable finishes the proof.

The bound \eqref{Al1} follows by a similar argument directly from the definition \eqref{Alex19}. 
\end{proof}

From this, using Proposition \ref{Alex3} and arguing as in \cite[Lemma 3.2]{IoPa}, one obtains the following lemma which essentially appears in \cite[Proposition 4.1]{HeTaTz}:

\begin{lemma}\label{NLEst2}
For $u_k\in X^1(I)$, $k=1\dots 5$, $|I|\leq 1$, the estimate
\begin{equation*}
\Vert \Pi_{i=1}^5\tilde{u}_k\Vert_{N(I)}\lesssim \sum_{\sigma\in\mathfrak{S}_5}\Vert u_{\sigma(1)}\Vert_{X^1(I)}\Pi_{j\ge 2}\Vert u_{\sigma(j)}\Vert_{Z^\prime(I)}
\end{equation*}
holds true, where $\tilde{u}_k\in\{u_k,\overline{u_k}\}$.
\end{lemma}

This follows from the more precise estimate
\begin{equation}\label{NLEst3}
\Vert \sum_{B\ge 1}P_{B}\tilde{u}_1\Pi_{j=2}^5P_{\le DB}\tilde{u}_j\Vert_{N(I)}\lesssim_D \Vert u_1\Vert_{X^1(I)}\Pi_{j=2}^5\Vert u_j\Vert_{Z^\prime(I)},
\end{equation}
which is proved similarly as in \cite[Lemma 3.2]{IoPa}. This implies the following:

\begin{proposition}[Local well-posedness]\label{LWP}
(i) Given $E>0$, there exists $\delta_0=\delta_0(E)>0$ such that if $\Vert \phi\Vert_{H^1(\mathbb{T}^3)}\le E$ and
\begin{equation*}
\Vert e^{it\Delta}\phi\Vert_{Z(I)}\le\delta_0
\end{equation*}
on some interval $I\ni 0$, $|I|\leq 1$, then there exists a unique solution $u\in X^1(I)$ of \eqref{eq1} satisfying $u(0)=\phi$. Besides
\begin{equation*}
\Vert u-e^{it\Delta}\phi\Vert_{X^1(I)}\lesssim_E \Vert e^{it\Delta}\phi\Vert_{Z(I)}^{3/2}.
\end{equation*}
The quantities $E(u)$ and $M(u)$ defined in \eqref{conserve} are conserved on $I$. 

(ii) If $u\in X^1(I)$ is a solution of \eqref{eq1.1} on some open interval $I$ and 
\begin{equation*}
\Vert u\Vert_{Z(I)}<+\infty
\end{equation*}
then $u$ can be extended as a nonlinear solution to a neighborhood of $\overline{I}$ and
\begin{equation*}
\Vert u\Vert_{X^1(I)}\le C(E(u),\Vert u\Vert_{Z(I)})
\end{equation*}
for some constant $C$ depending on $E(u)$ and $\Vert u\Vert_{Z(I)}$.
\end{proposition}

The main result in this section is the following:

\begin{proposition}[Stability]\label{Stabprop}
Assume $I$ is an open bounded interval, $\rho\in[-1,1]$, and $\widetilde{u}\in X^1(I)$ satisfies the approximate  Schr\"{o}dinger equation
\begin{equation}\label{ANLS}
(i\partial_t+\Delta)\widetilde{u}=\rho\widetilde{u}|\widetilde{u}|^4+e\quad\text{ on }\mathbb{T}^3\times I.
\end{equation}
Assume in addition that
\begin{equation}\label{ume}
\|\widetilde{u}\|_{Z(I)}+\|\widetilde{u}\|_{L^\infty_t(I,H^1(\mathbb{T}^3))}\leq M,
\end{equation}
for some $M\in[1,\infty)$. Assume $t_0 \in I$ and $u_0\in H^1(\mathbb{T}^3)$ is such that the smallness condition
\begin{equation}\label{safetycheck}
\|u_0 - \widetilde{u}(t_0)\|_{H^1(\mathbb{T}^3)}+\| e\|_{N(I)}\leq \eps
\end{equation}
holds for some $0 < \eps < \eps_1$, where $\eps_1\leq 1$ is a small constant $\eps_1 = \eps_1(M) > 0$.

Then there exists a strong solution $u\in X^1(I)$ of the Schr\"{o}dinger equation
\begin{equation}\label{ANLS2}
(i\partial_t+\Delta)u=\rho u|u|^4
\end{equation}
 such that $u(t_0)=u_0$ and
\begin{equation}\label{output}
\begin{split}
\| u \|_{X^1(I)}+\|\widetilde{u}\|_{X^1(I)}&\leq C(M),\\
\| u - \widetilde u \|_{X^1(I)}&\leq C(M)\eps.
 \end{split}
\end{equation}
\end{proposition}

The proof of these proposition is very similar to the proof of the corresponding statements in \cite[Section 3]{IoPa} and is omitted.


\section{Euclidean profiles}\label{Eucl}

In this section we prove precise estimates showing how to compare Euclidean and periodic solutions of both linear and nonlinear Schr\"{o}dinger equations. Such a comparison is meaningful only in the case of rescaled data that concentrate at a point, and then, only for short time (e.g. since the linear flow in $\mathbb{T}^3$ is periodic). We follow closely the arguments in \cite[Section 4]{IoPaSt}.

We fix a spherically-symmetric function $\eta\in C^\infty_0(\mathbb{R}^3)$ supported in the ball of radius $2$ and equal to $1$ in the ball of radius $1$. Given $\phi\in \dot{H}^1(\mathbb{R}^3)$ and a real number $N\geq 1$ we define
\begin{equation}\label{rescaled}
\begin{split}
Q_N\phi\in H^1(\mathbb{R}^3),\qquad &(Q_N\phi)(x)=\eta(x/N^{1/2})\phi(x),\\
\phi_N\in H^1(\mathbb{R}^3),\qquad &\phi_N(x)=N^\frac{1}{2}(Q_N\phi)(Nx),\\
f_{N}\in H^1(\T^3),\qquad &f_{N}(y)=\phi_N(\Psi^{-1}(y)),
\end{split}
\end{equation}
where $\Psi:\{x\in\mathbb{R}^3:|x|<1\}\to O_0\subseteq \T^3$, $\Psi(x)=x$. Thus $Q_N\phi$ is a compactly supported\footnote{This modification is useful to avoid the contribution of $\phi$ coming from the Euclidean infinity, in a uniform way depending on the scale $N$.} modification of the profile $\phi$, $\phi_N$ is an $\dot{H}^1$-invariant rescaling of $Q_N\phi$, and $f_{N}$ is the function obtained by transferring $\phi_N$ to a neighborhood of $0$ in $\T^3$. We define also
\begin{equation*}
E_{\mathbb{R}^3}(\phi)=\frac{1}{2}\int_{\mathbb{R}^3}|\nabla_{\R^3}\phi|^2\,dx+\frac{1}{6}\int_{\mathbb{R}^3}|\phi|^6\,dx.
\end{equation*}

We will use the main theorem of \cite{CKSTTcrit}, in the following form.

\begin{theorem}\label{MainThmEucl}
Assume $\psi\in\dot{H}^1(\mathbb{R}^3)$. Then there is a unique global solution $v\in C(\mathbb{R}:\dot{H}^1(\mathbb{R}^3))$ of the initial-value problem
\begin{equation}\label{clo3}
(i\partial_t+\Delta_{\R^3})v=v|v|^4,\qquad v(0)=\psi,
\end{equation}
and
\begin{equation}\label{clo4}
\|\,|\nabla_{\R^3} v|\,\|_{(L^\infty_tL^2_x\cap L^2_tL^6_x)(\mathbb{R}^3\times\mathbb{R})}\leq \widetilde{C}(E_{\mathbb{R}^3}(\psi)).
\end{equation}
Moreover this solution scatters in the sense that there exists $\psi^{\pm\infty}\in\dot{H}^1(\mathbb{R}^3)$ such that
\begin{equation}\label{EScat}
\Vert v(t)-e^{it\Delta_{\mathbb{R}^3}}\psi^{\pm\infty}\Vert_{\dot{H}^1(\mathbb{R}^3)}\to 0
\end{equation}
as $t\to\pm\infty$. Besides, if $\psi\in H^5(\mathbb{R}^3)$ then $v\in C(\mathbb{R}:H^5(\mathbb{R}^3))$ and
\begin{equation*}
\sup_{t\in\mathbb{R}}\|v(t)\|_{H^5(\mathbb{R}^3)}\lesssim_{\|\psi\|_{H^5(\mathbb{R}^3)}}1.
\end{equation*}
\end{theorem}

Our first result in this section is the following lemma:

\begin{lemma}\label{step1}
Assume $\phi\in\dot{H}^1(\mathbb{R}^3)$, $T_0\in(0,\infty)$, and $\rho\in\{0,1\}$ are given, and define $f_{N}$ as in \eqref{rescaled}. Then the following conclusions hold:

(i) There is $N_0=N_0(\phi,T_0)$ sufficiently large such that for any $N\geq N_0$ there is a unique solution $U_{N}\in C((-T_0N^{-2},T_0N^{-2}):H^1(\T^3))$ of the initial-value problem
\begin{equation}\label{clo5}
(i\partial_t+\Delta)U_N=\rho U_N|U_N|^4,\qquad U_N(0)=f_N.
\end{equation}
Moreover, for any $N\geq N_0$,
\begin{equation}\label{clo6}
\|U_N\|_{X^1(-T_0N^{-2},T_0N^{-2})}\lesssim_{E_{\mathbb{R}^3}(\phi)}1.
\end{equation}

(ii) Assume $\varepsilon_1\in(0,1]$ is sufficiently small (depending only on $E_{\mathbb{R}^3}(\phi)$), $\phi'\in H^5(\mathbb{R}^3)$, and $\|\phi-\phi'\|_{\dot{H}^1(\mathbb{R}^3)}\leq\varepsilon_1$. Let $v'\in C(\mathbb{R}:H^5)$ denote the solution of the initial-value problem
\begin{equation*}
(i\partial_t+\Delta_{\R^3})v'=\rho v'|v'|^4,\qquad v'(0)=\phi'.
\end{equation*}
For $R,N\geq 1$ we define
\begin{equation}\label{clo9}
\begin{split}
&v'_R(x,t)=\eta(x/R)v'(x,t),\qquad\,\,\qquad (x,t)\in\mathbb{R}^3\times(-T_0,T_0),\\
&v'_{R,N}(x,t)=N^\frac{1}{2}v'_R(Nx,N^2t),\qquad\quad\,(x,t)\in\mathbb{R}^3\times(-T_0N^{-2},T_0N^{-2}),\\
&V_{R,N}(y,t)=v'_{R,N}(\Psi^{-1}(y),t)\qquad\quad\,\, (y,t)\in\T^3\times(-T_0N^{-2},T_0N^{-2}).
\end{split}
\end{equation}
Then there is $R_0\geq 1$ (depending on $T_0$ and $\phi'$ and $\varepsilon_1$) such that, for any $R\geq R_0$,
\begin{equation}\label{clo18}
\limsup_{N\to\infty}\|U_N-V_{R,N}\|_{X^1(-T_0N^{-2},T_0N^{-2})}\lesssim_{E_{\mathbb{R}^3}(\phi)}\varepsilon_1.
\end{equation}
\end{lemma}

The proof of Lemma \ref{step1} is very similar to the proof of \cite[Lemma 4.2]{IoPaSt} and is omitted. To understand linear and nonlinear evolutions beyond the Euclidean window we need an additional extinction lemma:

\begin{lemma}\label{Extinction}
Let $\phi\in \dot{H}^1(\mathbb{R}^3)$ and define $f_N$ as in \eqref{rescaled}. For any $\varepsilon>0$, there exists $T=T(\psi,\varepsilon)$ and $N_0(\psi,\varepsilon)$ such that for all $N\ge N_0$, there holds that
\begin{equation*}
\Vert e^{it\Delta}f_{N}\Vert_{Z(TN^{-2},T^{-1})}\lesssim\varepsilon.
\end{equation*}
\end{lemma}

\begin{proof}[Proof of Lemma \ref{Extinction}] For $M\geq 1$, we define
\begin{equation*}
K_M(x,t)=\sum_{\xi\in\mathbb{Z}^3}e^{-i[t\vert \xi\vert^2+x\cdot\xi]}\eta^3(\xi/M)=e^{it\Delta}P_{\le M}\delta_0.
\end{equation*}
We note from \cite[Lemma 3.18]{Bo2} that $K_M$ satisfies
\begin{equation}\label{no2}
|K_M(x,t)|\lesssim \Big[\frac{M}{\sqrt{q}(1+M|t/(2\pi)-a/q|^{1/2})}\Big]^3
\end{equation}
if
\begin{equation}\label{no3}
t/(2\pi)=a/q+\beta,\qquad q\in\{1,\ldots,M\},\,\,a\in\mathbb{Z},\,\,(a,q)=1,\,\,|\beta|\leq (Mq)^{-1}.
\end{equation}
From this, we conclude that for any $1\leq S\leq M$
\begin{equation}\label{ExtincEst}
\Vert K_M(x,t)\Vert_{L^\infty_{x,t}(\mathbb{T}^3\times [SM^{-2},S^{-1}])}\lesssim S^{-3/2}M^3.
\end{equation}
This follows directly from \eqref{no2} and Dirichlet's lemma, by considering the cases $|t|\in[SM^{-2},M^{-1}]$ and $|t|\in[M^{-1},S^{-1}]$.

In view of the Strichartz estimates in Theorem \ref{Stric}, to prove the lemma we may assume that $\phi\in C^\infty_0(\mathbb{R}^3)$. In this case, from the definition,
\begin{equation}\label{Alex10}
\Vert P_Kf_N\Vert_{L^1(\mathbb{T}^3)}\lesssim_{\phi} N^{-5/2},\quad \Vert P_Kf_N\Vert_{L^2(\mathbb{T}^3)}\lesssim_\phi (1+K/N)^{-10}N^{-1}.
\end{equation}

Using the Strichartz estimates in Theorem \ref{Stric}, we obtain, for $p\in[5,\infty]$,
\begin{equation}\label{Alex11}
\Vert e^{it\Delta}P_Kf_N\Vert_{L^p_{x,t}(\mathbb{T}^3\times[-1,1])}\lesssim_\phi K^{3/2-5/p}(1+K/N)^{-10}N^{-1}.
\end{equation}
Therefore, if $1\leq T\leq N$ and $p\in\{6,24\}$,
\begin{equation}\label{Alex12}
\sum_{K\notin [NT^{-1/100},NT^{1/100}]}K^{5-p/2}\Vert e^{it\Delta}P_Kf_N\Vert_{L^p_{x,t}(\mathbb{T}^3\times[-1,1])}^p\lesssim_\phi T^{-1/100}.
\end{equation}

To estimate the remaining sum over $K\in [NT^{-1/100},NT^{1/100}]$ we use the first bound in \eqref{Alex10} together with \eqref{ExtincEst} (with $M\approx \max (K,N)$, $S\approx T$). It follows that, for all $K$,
\begin{equation}\label{Alex13}
\Vert e^{it\Delta}P_Kf_N\Vert_{L^\infty_{x,t}(\mathbb{T}^3\times [TN^{-2},T^{-1}])}\lesssim_\phi T^{-3/2}(K+N)^3N^{-5/2}.
\end{equation}
Interpolating with \eqref{Alex11}, for $p\in[5,\infty]$ and $K\in [NT^{-1/100},NT^{1/100}]$
\begin{equation}\label{Alex14}
\Vert e^{it\Delta}P_Kf_N\Vert_{L^p_{x,t}(\mathbb{T}^3\times [TN^{-2},T^{-1}])}\lesssim_\phi T^{-1+5/p}N^{1/2-5/p}.
\end{equation}
The lemma follows using \eqref{Alex12} and \eqref{Alex14}, by setting $T=T(\varepsilon,\psi)$ sufficiently large.
\end{proof}

For later use we record one more estimate that follows from \eqref{Alex13} and \eqref{Alex10}: if $\phi\in C^\infty_0(\R^3)$, $p\in [4,\infty]$, $1\leq T\leq N$, and $f_N$ is defined as in \eqref{rescaled}, then
\begin{equation}\label{ELiL1}
\sup_{|t|\in [TN^{-2},T^{-1}]}\|e^{it\Delta}f_N\|_{L^p(\T^3)}\lesssim_\phi T^{-1/10}N^{1/2-3/p}.
\end{equation}

We conclude this section with a proposition describing nonlinear solutions of the initial-value problem \eqref{eq1.1} corresponding to data concentrating at a point. In view of the profile analysis in the next section, we need to consider slightly more general data. Given $f\in L^2(\T^3)$, $t_0\in\mathbb{R}$ and $x_0\in\T^3$ we define
\begin{equation*}
\begin{split}
&(\pi_{x_0}f)(x):=f(x-x_0),\\
&(\Pi_{t_0,x_0})f(x)=(e^{-it_0\Delta}f)(x-x_0)=(\pi_{x_0}e^{-it_0\Delta}f)(x).
\end{split}
\end{equation*}
As in \eqref{rescaled}, given $\phi\in\dot{H}^1(\mathbb{R}^3)$ and $N\geq 1$, we define
\begin{equation*}
T_N\phi(x):=N^\frac{1}{2}\widetilde{\phi}(N\Psi^{-1}(x))\qquad\text{ where }\qquad\widetilde{\phi}(y):=\eta(y/N^{1/2})\phi(y),
\end{equation*}
and observe that
\begin{equation*}
T_N:\dot{H}^1(\mathbb{R}^3)\to H^1(\T^3)\text{ is a linear operator with }\|T_N\phi\|_{H^1(\T^3)}\lesssim \|\phi\|_{\dot{H}^1(\mathbb{R}^3)}.
\end{equation*}
Let $\widetilde{\mathcal{F}}_e$ denote the set of renormalized Euclidean frames
\begin{equation*}
\begin{split}
\widetilde{\mathcal{F}}_e:=\{(N_k,t_k,x_k)_{k\geq 1}:&\,N_k\in[1,\infty),\,t_k\to 0,\,x_k\in\T^3,\\
&\,N_k\to\infty,\text{ and }t_k=0 \text{ or }N_k^2|t_k|\to\infty\}.
\end{split}
\end{equation*}

\begin{proposition}\label{GEForEP}
Assume that $\mathcal{O}=(N_k,t_k,x_k)_k\in\widetilde{\mathcal{F}}_e$, $\phi\in\dot{H}^1(\mathbb{R}^3)$, and let $U_k(0)=\Pi_{t_k,x_k}(T_{N_k}\phi)$.

(i) There exists $\tau=\tau(\phi)$ such that for $k$ large enough (depending only on $\phi$ and $\mathcal{O}$) there is a nonlinear solution $U_k\in X^1(-\tau,\tau)$ of the initial-value problem \eqref{eq1.1} and
\begin{equation}\label{ControlOnZNormForEP}
\Vert U_k\Vert_{X^1(-\tau,\tau)}\lesssim_{E_{\mathbb{R}^3}(\phi)}1.
\end{equation}

(ii) There exists a Euclidean solution $u\in C(\mathbb{R}:\dot{H}^1(\mathbb{R}^3))$ of
\begin{equation}\label{EEq}
\left(i\partial_t+\Delta_{\R^3}\right)u=u\vert u\vert^4
\end{equation}
with scattering data $\phi^{\pm\infty}$ defined as in \eqref{EScat} such that the following holds, up to a subsequence:
for any $\varepsilon>0$, there exists $T(\phi,\varepsilon)$ such that for all $T\ge T(\phi,\varepsilon)$ there exists $R(\phi,\varepsilon,T)$ such that for all $R\ge R(\phi,\varepsilon,T)$, there holds that
\begin{equation}\label{ProxyEuclHyp}
\Vert U_k-\tilde{u}_k\Vert_{X^1(\{\vert t-t_k\vert\le TN_k^{-2}\}\cap\{\vert t\vert\le T^{-1}\})}\le\varepsilon,
\end{equation}
for $k$ large enough, where
\begin{equation*}
(\pi_{-x_k}\tilde{u}_k)(x,t)=N_k^\frac{1}{2}\eta(N_k\Psi^{-1}(x)/R)u(N_k\Psi^{-1}(x),N_k^2(t-t_k)).
\end{equation*}
In addition, up to a subsequence,
\begin{equation}\label{ScatEuclSol}
\Vert U_k(t)-\Pi_{t_k-t,x_k}T_{N_k}\phi^{\pm\infty}\Vert_{X^1(\{\pm(t-t_k)\geq TN_k^{-2}\}\cap \{\vert t \vert\le T^{-1}\})}\le \varepsilon,
\end{equation}
for $k$ large enough (depending on $\phi,\varepsilon,T,R$).
\end{proposition}

\begin{proof}[Proof of Proposition \ref{GEForEP}] Clearly, we may assume that $x_k=0$.

We have two cases. If $t_k=0$ for any $k$ then the lemma follows from Theorem \ref{MainThmEucl}, Lemma \ref{step1} and Lemma \ref{Extinction}: we let $u$ be the nonlinear Euclidean solution of \eqref{EEq} with $u(0)=\phi$ and notice that for any $\delta>0$ there is $T(\phi,\delta)$ such that
\begin{equation*}
\|\nabla_{\R^3} u\|_{L^\frac{10}{3}_{x,t}(\mathbb{R}^3\times\{|t|\geq T(\phi,\delta)\})}\leq\delta.
\end{equation*}
The bound \eqref{ProxyEuclHyp} follows for any fixed $T\geq T(\phi,\delta)$ from Lemma \ref{step1}. Assuming $\delta$ is sufficiently small and $T$ is sufficiently large (both depending on $\phi$ and $\varepsilon$), the bound \eqref{ScatEuclSol} then follow from Theorem \ref{MainThmEucl}, Lemma \ref{step1} and Lemma \ref{Extinction} (which guarantee smallness of $\mathbf{1}_{\pm}(t)\cdot e^{it\Delta}U_k(\pm N_k^{-2}T(\phi,\delta))$ in $Z(\{\vert t\vert\le T^{-1}\})$) and Proposition \ref{LWP}.

Otherwise, if $\lim_{k\to\infty}N_k^2|t_k|=\infty$, we may assume by symmetry that $N_k^2t_k\to+\infty$. Then we let $u$ be the solution of
\eqref{EEq} such that
\begin{equation*}
\Vert\nabla_{\R^3}\left(u(t)-e^{it\Delta_{\R^3}}\phi\right)\Vert_{L^2(\mathbb{R}^3)}\to0
\end{equation*}
as $t\to-\infty$ (thus $\phi^{-\infty}=\phi$).
We let $\tilde{\phi}=u(0)$ and  apply the conclusions of the lemma to the frame $(N_k,0,0)_k\in\mathcal{F}_e$ and $V_k(s)$, the solution of \eqref{eq1} with initial data $V_k(0)=T_{N_k}\tilde{\phi}$. In particular, we see from the fact that $N_k^2t_k\to+\infty$ and \eqref{ScatEuclSol} that
\begin{equation*}
\Vert V_k(-t_k)-\Pi_{t_k,0}T_{N_k}\phi\Vert_{H^1(\T^3)}\to 0
\end{equation*}
as $k\to\infty$. Then, using Proposition \ref{Stabprop}, we see that
\begin{equation*}
\Vert U_k-V_k(\cdot-t_k)\Vert_{X^1(-T^{-1},T^{-1})}\to 0
\end{equation*}
as $k\to\infty$, and we can conclude by inspecting the behavior of $V_k$. This ends the proof.
\end{proof}


\section{Profile decompositions}\label{profiles}

In this section we show that given a bounded sequence of functions $f_k\in H^1(\T^3)$ we can construct suitable {\it{profiles}} and express the sequence in terms of these profiles. The statements and the arguments in this section are very similar to those in \cite[Section 5]{IoPaSt} and \cite[Section 5]{IoPa}. See also \cite{Ker} for the original proofs of Keraani in the Euclidean geometry.

As before, given $f\in L^2(\R^3)$, $t_0\in\mathbb{R}$, and $x_0\in\T^3$ we define
\begin{equation}\label{PI}
\begin{split}
&(\pi_{x_0}f)(x):=f(x-x_0),\\
&(\Pi_{t_0,x_0})f(x)=(e^{-it_0\Delta}f)(x-x_0)=(\pi_{x_0}e^{-it_0\Delta}f)(x).
\end{split}
\end{equation}
As in \eqref{rescaled}, given $\phi\in\dot{H}^1(\mathbb{R}^3)$ and $N\geq 1$, we define
\begin{equation}\label{TN}
T_N\phi(x):=N^\frac{1}{2}\widetilde{\phi}(N\Psi^{-1}(x))\qquad\text{ where }\qquad\widetilde{\phi}(y):=\eta(y/N^{1/2})\phi(y),
\end{equation}
and observe that
\begin{equation}\label{TN2}
T_N:\dot{H}^1(\mathbb{R}^3)\to H^1(\T^3)\text{ is a linear operator with }\|T_N\phi\|_{H^1(\T^3)}\lesssim \|\phi\|_{\dot{H}^1(\mathbb{R}^3)}.
\end{equation}

The following is our main definition.

\begin{definition}\label{DefPro}

\begin{enumerate}

\item We define a Euclidean frame to be a sequence $\mathcal{F}_e=(N_k,t_k,x_k)_k$ with $N_k\ge 1$, $N_k\to+\infty$, $t_k\in\mathbb{R}$, $t_k\to 0$, $x_k\in\mathbb{T}^3$. We say that two frames $(N_k,t_k,x_k)_k$ and $(M_k,s_k,y_k,)_k$ are orthogonal if
\begin{equation*}
\lim_{k\to+\infty} \left(\left\vert \ln\frac{N_k}{M_k}\right\vert+N_k^2\vert t_k-s_k\vert+N_k\vert x_k-y_k\vert\right)=+\infty.\end{equation*}
Two frames that are not orthogonal are called equivalent.

\item If $\mathcal{O}=(N_k,t_k,x_k)_k$ is a Euclidean frame and if $\phi\in \dot{H}^1(\mathbb{R}^3)$, we define the Euclidean profile associated to $(\phi,\mathcal{O})$ as the sequence $\widetilde{\phi}_{\mathcal{O}_k}$
\begin{equation*}
\widetilde{\phi}_{\mathcal{O}_k}(x):=\Pi_{t_k,x_k}(T_{N_k}\phi).
\end{equation*}
\end{enumerate}
\end{definition}

The following lemma summarizes some of the basic properties of profiles associated to equivalent/orthogonal frames. Its proof uses Lemma \ref{step1} with $\rho=0$ to control linear evolutions inside the Euclidean window and the bound \eqref{ELiL1} to control these evolutions outside such a window. Given these ingredients, the proof of Lemma \ref{EquivFrames} is very similar to the proof of Lemma 5.4 in \cite{IoPaSt}, and is omitted.

\begin{lemma}(Equivalence of frames)\label{EquivFrames}

(i) If $\mathcal{O}$ and $\mathcal{O}^\prime$ are equivalent Euclidean profiles, then, there exists an isometry of $\dot{H}^1(\mathbb{R}^3)$, $T$ such that for any profile $\widetilde{\psi}_{\mathcal{O}^\prime_k}$, up to a subsequence there holds that
\begin{equation}\label{equiv}
\limsup_{k\to+\infty}
\Vert \widetilde{T\psi}_{\mathcal{O}_k}-\widetilde{\psi}_{\mathcal{O}^\prime_k}\Vert_{H^1(\mathbb{T}^3)}=0.
\end{equation}

(ii) If $\mathcal{O}$ and $\mathcal{O}^\prime$ are orthogonal frames and $\widetilde{\psi}_{\mathcal{O}_k}$, $\widetilde{\varphi}_{\mathcal{O}^\prime_k}$ are corresponding profiles, then, up to a subsequence,
\begin{equation*}
\begin{split}
\lim_{k\to+\infty}\langle \widetilde{\psi}_{\mathcal{O}_k},\widetilde{\varphi}_{\mathcal{O}^\prime_k}\rangle_{H^1\times H^1(\mathbb{T}^3)}&=0,\\
\lim_{k\to+\infty}\langle |\widetilde{\psi}_{\mathcal{O}_k}|^3,|\widetilde{\varphi}_{\mathcal{O}^\prime_k}|^3\rangle_{L^2\times L^2(\mathbb{T}^3)}&=0.
\end{split}
\end{equation*}

(iii) If $\mathcal{O}$ is a Euclidean frame and $\widetilde{\psi}_{\mathcal{O}_k}$, $\widetilde{\varphi}_{\mathcal{O}_k}$ are two profiles corresponding to $\mathcal{O}$, then
\begin{equation*}
\begin{split}
&\lim_{k\to+\infty}\left(\Vert\widetilde{\psi}_{\mathcal{O}_k}\Vert_{L^2}+\Vert\widetilde{\varphi}_{\mathcal{O}_k}\Vert_{L^2}\right)=0,\\
&\lim_{k\to+\infty}\langle \widetilde{\psi}_{\mathcal{O}_k},\widetilde{\varphi}_{\mathcal{O}_k}\rangle_{H^1\times H^1(\mathbb{T}^3)}=\langle \psi,\varphi\rangle_{\dot{H}^1\times\dot{H}^1(\mathbb{R}^3)}.
\end{split}
\end{equation*}
\end{lemma}

\begin{definition}\label{absent}
We say that a sequence of functions $\{f_k\}_k\subseteq H^1(\T^3)$ is absent from a frame $\mathcal{O}$ if, up to a subsequence, for every profile $\psi_{\mathcal{O}_k}$ associated to $\mathcal{O}$,
\begin{equation*}
\int_{\mathbb{T}^3}\left(f_k\overline{\widetilde{\psi}}_{\mathcal{O}_k}+\nabla f_k\nabla\overline{\widetilde{\psi}}_{\mathcal{O}_k}\right)dx\to0
\end{equation*}
as $k\to+\infty$.
\end{definition}

Note in particular that a profile associated to a frame $\mathcal{O}$ is absent from any frame orthogonal to $\mathcal{O}$.

The following proposition is the core of this section. Its proof is similar to the proof of \cite[Proposition 5.5]{IoPa}, and is omitted.

\begin{proposition}\label{PD}
Consider $\{f_k\}_k$ a sequence of functions in $H^1(\mathbb{T}^3)$ satisfying
\begin{equation}\label{FkBoundedPD}
\limsup_{k\to+\infty}\Vert f_k\Vert_{H^1(\mathbb{T}^3)}\lesssim E
\end{equation}
and a sequence of intervals $I_k=(-T_k,T^k)$ such that $\vert I_k\vert\to0$ as $k\to+\infty$\footnote{The condition \eqref{smallnessPD} on the smallness of the remainder $R_k^J$ depends on both the sequence of functions $f_k$ and the sequence of intervals $I_k$. The existence of both these sequences is a consequence of the contradiction assumption $E_{max}<\infty$ in Theorem \ref{Alex40}.}. Up to passing to a subsequence, assume that $f_k\rightharpoonup g\in H^1(\mathbb{T}^3)$.
There exists a sequence of profiles $\widetilde{\psi}^\alpha_{\mathcal{O}^\alpha_k}$ associated to pairwise orthogonal Euclidean frames $\mathcal{O}^\alpha$ such that, after extracting a subsequence, for every $J\ge 0$
\begin{equation}\label{DecompositionPD}
f_k=g+\sum_{1\le \alpha\le J}\widetilde{\psi}^\alpha_{\mathcal{O}^\alpha_k}+R_k^J
\end{equation}
where $R_k^J$ is absent from the frames $\mathcal{O}^\alpha$, $\alpha\le J$ and is small in the sense that
\begin{equation}\label{smallnessPD}
\limsup_{J\to+\infty}\limsup_{k\to+\infty}\big[\sup_{N\ge 1,t\in I_k,\,x\in\mathbb{T}^3}N^{-\frac{1}{2}}\left\vert \left(e^{it\Delta}P_NR_k^J\right)(x)\right\vert\big]=0.
\end{equation}
Besides, we also have the following orthogonality relations
\begin{equation}\label{OrthogonalityPD}
\begin{split}
&\Vert f_k\Vert_{L^2}^2=\Vert g\Vert_{L^2}^2+\Vert R_k^J\Vert_{L^2}^2+o_k(1),\\
&\Vert \nabla f_k\Vert_{L^2}^2=\Vert \nabla g\Vert_{L^2}^2+\sum_{\alpha\le J}\Vert\nabla_{\R^3}\psi^\alpha\Vert_{L^2(\mathbb{R}^3)}^2+\Vert\nabla R_k^J\Vert_{L^2}^2+o_k(1),\\
&\lim_{J\to+\infty}\limsup_{k\to+\infty}\left\vert\Vert f_k\Vert_{L^6}^6-\Vert g\Vert_{L^6}^6-\sum_{\alpha\le J}\Vert\widetilde{\varphi}^\alpha_{\mathcal{O}^\alpha_k}\Vert_{L^6}^6\right\vert=0,
\end{split}
\end{equation}
where $o_k(1)\to0$ as $k\to+\infty$, possibly depending on $J$.
\end{proposition}


\section{Proof of the main theorem}\label{proofthm}

From Proposition \ref{LWP}, we see that to prove Theorem \ref{Main1}, it suffices to prove that solutions remain bounded in $Z$ on intervals of length at most $1$. To obtain this, we induct on the energy $E(u)$.

Define
\begin{equation*}
\Lambda(L,\tau)=\sup\{\Vert u\Vert_{Z(I)}^2,E(u)\le L,\vert I\vert\le \tau\}
\end{equation*}
where the supremum is taken over all strong solutions of
\eqref{eq1} of energy less than or equal to $L$ and all intervals $I$ of length $\vert I\vert\le \tau$. Clearly, $\Lambda$ is an increasing function of both its arguments and moreover,
\begin{equation*}
\Lambda(L,\tau+\sigma)\lesssim\Lambda(L,\tau)+\Lambda(L,\sigma).
\end{equation*}
Hence we may define
\begin{equation*}
\Lambda_\ast(L)=\lim_{\tau\to 0}\Lambda(L,\tau)
\end{equation*}
and we have that for all $\tau$,
\begin{equation*}
\Lambda(L,\tau)<+\infty\Leftrightarrow \Lambda_\ast(L)<+\infty.
\end{equation*}
Finally, we define
\begin{equation}\label{Emax}
E_{max}=\sup\{L: \Lambda_\ast(L)<+\infty\}.
\end{equation}
We see that Theorem \ref{Main1} is equivalent to the following statement.

\begin{theorem}\label{Alex40}
$E_{max}=+\infty$. In particular every solution of \eqref{eq1} is global.
\end{theorem}

\begin{proof}[Proof of Theorem \ref{Alex40}] Suppose for contradiction that $E_{max}<+\infty$. From now on, all our constants are allowed to depend on $E_{max}$. By definition, there exists a sequence of solutions $u_k$ such that
\begin{equation}\label{CondForComp}
E(u_k)\to E_{max},\quad \Vert u_k\Vert_{Z(-T_k,0)},\Vert u_k\Vert_{Z(0,T^k)}\to+\infty
\end{equation}
for some $T_k,T^k\to0$ as $k\to+\infty$.
We now apply Proposition \ref{PD} to the sequence $\{u_k(0)\}_k$ with $I_k=(-T_k,T^k)$. This gives a decomposition
\begin{equation*}
u_k(0)=g+\sum_{1\le\alpha\le J}\widetilde{\psi}^\alpha_{\mathcal{O}^\alpha_k}+R^J_k.
\end{equation*}
We first consider the remainder and note that, for $p\in\{p_0,p_1\}$ and $q=(p_0+4)/2$,
\begin{equation*}
\begin{split}
\sum_N N^{5-p/2}&\Vert P_Ne^{it\Delta}R^J_k\Vert_{L^{p}_{t,x}(\T^3\times I_k)}^{p}\\
&\lesssim \big[\sup_NN^{-\frac{1}{2}}\Vert e^{it\Delta}P_N R^J_k\Vert_{L^\infty_{t,x}(\mathbb{T}^3\times I_k)}\big]^{p-q}\sum_N \left[N^{5/q-1/2}\Vert P_Ne^{it\Delta}R^J_k\Vert_{L^q_{t,x}(\mathbb{T}^3\times I_k)}\right]^q\\
&\lesssim \big[\sup_NN^{-\frac{1}{2}}\Vert e^{it\Delta}P_N R^J_k\Vert_{L^\infty_{t,x}(\mathbb{T}^3\times I_k)}\big]^{p-q}\sum_N N^q\Vert P_N R^J_k\Vert_{L^2}^q\\
&\lesssim\big[\sup_NN^{-\frac{1}{2}}\Vert e^{it\Delta}P_N R^J_k\Vert_{L^\infty_{t,x}(\mathbb{T}^3\times I_k)}\big]^{p-q}.
\end{split}
\end{equation*}
Therefore
\begin{equation}\label{SmallnessRterm}
\limsup_{J\to+\infty}\limsup_{k\to+\infty}\Vert e^{it\Delta}R^J_k\Vert_{Z(I_k)}=0.
\end{equation}

\medskip

{\bf Case I:} $\{u_k(0)\}_k$ converges strongly in $H^1(\mathbb{T}^3)$ to its limit $g$ which satisfies $E(g)=E_{max}$. Then, by Strichartz estimates, there exists $\eta>0$ such that, for $k$ large enough
\begin{equation*}
\Vert e^{it\Delta}u_k(0)\Vert_{Z(-T_k,T^k)}\le\Vert e^{it\Delta}g\Vert_{Z(-\eta,\eta)}+o_k(1)\le\delta_0,
\end{equation*}
where $\delta_0$ is given by the local theory in Proposition \ref{LWP}. In this case, we conclude that $\Vert u_k\Vert_{Z(-T_k,T^k)}\lesssim 2\delta_0$ which contradicts \eqref{CondForComp}.

\medskip

{\bf Case IIa:} $g=0$ and there are no profile. Then, taking $J$ sufficiently large, we get that
\begin{equation*}
\Vert e^{it\Delta}u_k(0)\Vert_{Z(I_k)}=\Vert e^{it\Delta}R^J_k\Vert_{Z(I_k)}\le\delta_0
\end{equation*}
where $\delta_0$ is as above. Once again, this contradicts \eqref{CondForComp}.

\medskip

{\bf Case IIb:} $g=0$ and there is only one Euclidean profile, such that
\begin{equation*}
u_k(0)=\widetilde{\psi}_{\mathcal{O}_k}+o_k(1)
\end{equation*}
in $H^1$ (see \eqref{SumOfL}), where $\mathcal{O}$ is a Euclidean frame. In this case, we let $U_k$ be the solution of \eqref{eq1} with initial data $U_k(0)=\widetilde{\psi}_{\mathcal{O}_k}$ and we use \eqref{ControlOnZNormForEP} to get, for $k$ large enough
\begin{equation*}
\Vert U_k\Vert_{Z(-T_k,T^k)}\le\Vert U_k\Vert_{Z(-\delta,\delta)}\lesssim 1\quad\text{and}\quad\lim_{k\to +\infty}\Vert U_k(0)-u_k(0)\Vert_{H^1}\to 0.
\end{equation*}
We may use Proposition \ref{Stabprop} to deduce that
\begin{equation*}
\Vert u_k\Vert_{Z(-T_k,T^k)}\lesssim \Vert u_k\Vert_{X^1(-T_k,T^k)}\lesssim 1
\end{equation*}
which contradicts \eqref{CondForComp}.

\medskip

{\bf Case III:} There exists at least one profile or $g\ne 0$. Using Lemma \ref{EquivFrames} and passing to a subsequence, we may renormalize every Euclidean profile, that is, up to passing to an equivalent profile, we may assume that for every Euclidean frame $\mathcal{O}^\alpha$, $\mathcal{O}^\alpha\in\widetilde{\mathcal{F}}_e$.
Besides, using Lemma \ref{EquivFrames} and passing to a subsequence once again, we may assume that for every $\alpha\ne\beta$,
either $N^\alpha_k/N^\beta_k+N^\beta_k/N^\alpha_k\to+\infty$ as $k\to+\infty$ or $N^\alpha_k=N^\beta_k$ for all $k$ and in this case, either $t^\alpha_k=t^\beta_k$ as $k\to+\infty$ or $(N^\alpha_k)^2\vert t^\alpha_k-t^\beta_k\vert \to+\infty$ as $k\to+\infty$.
Now for every linear profile $\widetilde{\psi}^\alpha_{\mathcal{O}^\alpha_k}$, we define the associated nonlinear profile $U^\alpha_k$ as the maximal solution of \eqref{eq1} with initial data $U^\alpha_k(0)=\widetilde{\psi}^\alpha_{\mathcal{O}^\alpha_k}$. A more precise description of each nonlinear profile is given by Proposition \ref{GEForEP}. Similarly, we define $W$ to be the nonlinear solution of \eqref{eq1} with initial data $g$.

From \eqref{OrthogonalityPD} we see that, after extracting a subsequence,
\begin{equation}\label{SumOfL}
\begin{split}
&E(\alpha):=\lim_{k\to+\infty}E(\widetilde{\psi}^\alpha_{\mathcal{O}^\alpha_k})\in(0,E_{max}],\\
&\lim_{J\to+\infty}\big[\sum_{1\le\alpha\le J}E(\alpha)+\lim_{k\to+\infty}E(R_k^J)\big]\le E_{max}-E(g).
\end{split}
\end{equation}
Up to relabeling the profiles, we can assume that for all $\alpha$, $E(\alpha)\le E(1)<E_{max}-\eta$, $E(g)<E_{max}-\eta$ for some $\eta>0$. Consequently, all the nonlinear profiles are global and satisfy
\begin{equation*}
\Vert W\Vert_{Z(-1,1)}+\Vert U^\alpha_k\Vert_{Z(-1,1)}\le 3\Lambda(E_{max}-\eta/2,2)\lesssim 1,
\end{equation*}
where from now on all the implicit constants are allowed to depend on $\Lambda(E_{max}-\eta/2,2)$. Using Proposition \ref{Stabprop} it follows that
\begin{equation}\label{BddX1}
\Vert W\Vert_{X^1(-1,1)}+\Vert U^\alpha_k\Vert_{X^1(-1,1)}\lesssim 1.
\end{equation}

For $J,k\geq 1$ we define
\begin{equation*}
U^J_{prof,k}:=W+\sum_{\al=1}^J U^\al_k.
\end{equation*}
We show first that there is a constant $Q\lesssim 1$ such that
\begin{equation}\label{bi1}
\Vert U^J_{prof,k}\Vert_{X^1(-1,1)}^2+\Vert W\Vert_{X^1(-1,1)}^2+\sum_{\al=1}^J\|U^\al_k\|_{X^1(-1,1)}^2+\sum_{\alpha=1}^J\Vert U^\alpha_k-e^{it\Delta}\widetilde{\psi}^\alpha_{\mathcal{O}^\alpha_k}\Vert_{X^1(-1,1)}\leq Q^2,
\end{equation}
uniformly in $J$, for all $k\ge k_0(J)$ sufficiently large. Indeed, a simple fixed point argument as in section \ref{localwp} shows that there exists $\delta_0>0$ such that if
\begin{equation*}
\Vert \phi\Vert_{H^1(\mathbb{T}^3)}=\delta\le\delta_0
\end{equation*}
then the unique strong solution of \eqref{eq1} with initial data $\phi$ is global and satisfies
\begin{equation}\label{SmalldataCCL}
\begin{split}
\Vert u\Vert_{X^1(-2,2)}&\le 2\delta\quad\text{and}\quad\Vert u-e^{it\Delta}\phi\Vert_{X^1(-2,2)}\lesssim \delta^4.
\end{split}
\end{equation}
From \eqref{SumOfL}, we know that there are only finitely many profiles such that $E(\alpha)\geq\delta_0/2$. Without loss of generality, we may assume that for all $\alpha\ge A$, $E(\alpha)\leq\delta_0$. Using \eqref{OrthogonalityPD}, \eqref{BddX1}, and \eqref{SmalldataCCL} we then see that
\begin{equation*}
\begin{split}
&\Vert U^J_{prof,k}\Vert_{X^1(-1,1)}=\Vert  W+\sum_{1\le\alpha\le J}U^\alpha_k\Vert_{X^1(-1,1)}\\
&\le \Vert W\Vert_{X^1(-1,1)}+\sum_{1\le\alpha\le A}\Vert U^\alpha_k\Vert_{X^1(-1,1)}+\Vert \sum_{A\le\alpha\le J}(U^\alpha_k-e^{it\Delta}U^\alpha_k(0))\Vert_{X^1(-1,1)}\\
&+\Vert e^{it\Delta}\sum_{A\le\alpha\le J}U^\alpha_k(0)\Vert_{X^1(-1,1)}\\
&\lesssim 1+A+\sum_{A\le\alpha\le J}E(\alpha)+\Vert\sum_{A\le\alpha\le J}U^\alpha_k(0)\Vert_{H^1}\lesssim 1.
\end{split}
\end{equation*}
The bound on $\sum_{\al=1}^J\|U^\al_k\|_{X^1(-1,1)}^2$ is similar (in fact easier), which gives \eqref{bi1}.

We now claim that
\begin{equation*}
U^J_{app,k}=W+\sum_{1\le\alpha\le J}U^\alpha_k+e^{it\Delta}R^J_k
\end{equation*}
is an approximate solution for all $J\ge J_0$ and all $k\ge k_0(J)$ sufficiently large. We saw in \eqref{bi1} that $U^J_{app,k}$ has bounded $X^1$-norm. Let $\varepsilon=\varepsilon (2Q^2)$ be the constant given in Proposition \ref{Stabprop}. We compute
\begin{equation*}
\begin{split}
e&=\left(i\partial_t+\Delta\right)U^J_{app,k}-F(U^J_{app,k})=F(W)+\sum_{1\le\alpha\le J}F(U^\alpha_k)-F(U^J_{app,k})\\
&=F(U^J_{prof,k})-F(U^J_{prof,k}+e^{it\Delta}R^J_k)+F(W)+\sum_{1\le\alpha\le J}F(U^\alpha_k)-F(U^J_{prof,k}).
\end{split}
\end{equation*}
and appealing to Lemma \ref{AlmostSol} below, we obtain that
\begin{equation*}
\limsup_{k\to+\infty}\Vert e\Vert_{N(I_k)}\le\varepsilon/2
\end{equation*}
for $J\ge J_0(\varepsilon)$. In this case, we may use Proposition \ref{Stabprop} to conclude that $u_k$ satisfies
\begin{equation*}
\Vert u_k\Vert_{X^1(I_k)}\lesssim \Vert U^J_{app,k}\Vert_{X^1(I_k)}\le \Vert U^J_{prof,k}\Vert_{X^1(-1,1)}+\Vert e^{it\Delta}R^J_k\Vert_{X^1(-1,1)}\lesssim 1
\end{equation*}
which contradicts \eqref{CondForComp}. This finishes the proof.
\end{proof}

We have now proved our main theorem, except for the following important assertion.

\begin{lemma}\label{AlmostSol}

With the notation in {\bf Case III} of the proof of Theorem \ref{Alex40}, we have that, for fixed $J$,
\begin{equation}\label{Asol1}
\limsup_{k\to+\infty}\Vert F(U^J_{prof,k})-F(W)-\sum_{1\le\alpha\le J}F(U^\alpha_k)\Vert_{N(I_k)}=0.
\end{equation}
Besides, we also have that
\begin{equation}\label{Asol2}
\limsup_{J\to+\infty}\limsup_{k\to+\infty}\Vert F(U^J_{prof,k}+e^{it\Delta}R^J_k)-F(U^J_{prof,k})\Vert_{N(I_k)}=0.
\end{equation}
\end{lemma}


\section{Proof of lemma \ref{AlmostSol}}\label{ProofLastLem}

We will need the following lemma which states that a high-frequency linear solution does not interact significantly with a low-frequency profile. Recall from Section \ref{prelim} that $\mathfrak{O}_{4,1}(a,b)$ denotes a quentity which is quartic in $\{a,\overline{a}\}$ and linear in $\{b,\overline{b}\}$.

\begin{lemma}\label{HFLF}
Assume that $B,N\geq 2$ are dyadic numbers and $\omega:\T^3\times(-1,1)\to\mathbb{C}$ is a function satisfying $\vert \nabla^j\omega\vert \leq N^{j+1/2}\mathbf{1}_{\{|x|\leq N^{-1},\,|t|\leq N^{-2}\}}$, $j=0,1$. Then
\begin{equation*}
\Vert \mathfrak{O}_{4,1}(\omega,e^{it\Delta}P_{> BN}f)\Vert_{L^1((-1,1),H^1)}\lesssim (B^{-1/200}+N^{-1/200})\|f\|_{H^1(\T^3)}.
\end{equation*}
\end{lemma}

\begin{proof}[Proof of Lemma \ref{HFLF}]
We may assume that $\|f\|_{H^1(\T^3)}=1$ and $f=P_{>BN}f$. We notice that
\begin{equation*}
\begin{split}
\Vert \mathfrak{O}_{4,1}(\omega,e^{it\Delta}P_{>BN}f)\Vert_{L^1((-1,1),H^1)}&\lesssim \Vert \mathfrak{O}_{4,1}(\omega,\nabla e^{it\Delta}f)\Vert_{L^1((-1,1),L^2)}\\
&+\Vert e^{it\Delta}f\Vert_{L^\infty_tL^2}\Vert \omega\Vert_{L^4_tL^\infty_x}^3\Vert \vert\nabla\omega\vert+\vert\omega\vert\Vert_{L^4_tL^\infty_x}\\
&\lesssim \Vert \mathfrak{O}_{4,1}(\omega,\nabla e^{it\Delta}f)\Vert_{L^1((-1,1),L^2)}+B^{-1}.
\end{split}
\end{equation*}
Now we write
\begin{equation*}
\begin{split}
\Vert \mathfrak{O}_{4,1}(\omega,\nabla e^{it\Delta}f)\Vert_{L^1((-1,1),L^2)}^2&\lesssim N^{-2}\Vert W^\frac{1}{2}\nabla e^{it\Delta}f\Vert_{L^2(\mathbb{T}^3\times (-1,1))}^2\\
&\lesssim N^{-2}\sum_{j=1}^3\int_{-1}^1\langle e^{it\Delta}\partial_jf,We^{it\Delta}\partial_jf\rangle_{L^2\times L^2(\mathbb{T}^3)} dt\\
&\lesssim N^{-2}\sum_{j=1}^3\langle \partial_jf,\left[\int_{t=-1}^1e^{-it\Delta}We^{it\Delta}dt\right] \partial_jf\rangle_{L^2\times L^2(\mathbb{T}^3)},
\end{split}
\end{equation*}
where $W(x,t)=N^4\eta^3(N\Psi^{-1}(x))\eta^1(N^2t)$. Therefore, it remains to prove that
\begin{equation}\label{Alex50}
\|K\|_{L^2(\T^3)\to L^2(\T^3)}\lesssim N^2(B^{-1/100}+N^{-1/100})\,\,\text{ where }\,\,K=P_{>BN}\int_{\mathbb{R}}e^{-it\Delta}We^{it\Delta}P_{>BN}\,dt.
\end{equation}

We compute the Fourier coefficients of $K$ as follows
\begin{equation*}
\begin{split}
c_{p,q}&=\langle e^{ipx},Ke^{iqx}\rangle_{L^2\times L^2(\mathbb{T}^3)}\\
&=(1-\eta^3)(p/BN)(1-\eta^3)(q/BN)\int_{(-1,1)\times\mathbb{T}^3}e^{it\left[\vert p\vert^2-\vert q\vert^2\right]+i(q-p)\cdot x}W_k(x,t)dxdt\\
&=C\left(\mathcal{F}_{x,t}W\right)(p-q,\vert q\vert^2-\vert p\vert^2)(1-\eta^3)(p/BN)(1-\eta^3)(q/BN).
\end{split}
\end{equation*}
Hence, we obtain that
\begin{equation}\label{Estimcpq}
\vert c_{p,q}\vert\lesssim N^{-1}\left[1+\frac{\vert \vert p\vert^2-\vert q\vert^2\vert}{N^2}\right]^{-10}\left[1+\frac{\vert p-q\vert}{N}\right]^{-10}\mathbf{1}_{\{|p|\ge BN\}}\mathbf{1}_{\{|q|\ge BN\}}.
\end{equation}
Using Schur's lemma
\begin{equation*}
\|K\|_{L^2(\T^3)\to L^2(\T^3)}\lesssim \sup_{p\in\Z^3}\sum_{q\in\Z^3}|c_{p,q}|+\sup_{q\in\Z^3}\sum_{p\in\Z^3}|c_{p,q}|,
\end{equation*}
and the bound \eqref{Estimcpq}, for \eqref{Alex50} it suffices to prove that
\begin{equation}\label{Alex51}
N^{-3}\sup_{|p|\geq BN}\sum_{v\in\Z^3}\left[1+\frac{\vert \vert p\vert^2-\vert p+v\vert^2\vert}{N^2}\right]^{-10}\left[1+\frac{\vert v\vert}{N}\right]^{-10}\lesssim (B^{-1/100}+N^{-1/100}).
\end{equation}

We notice that the sum over $|v|\geq N\min(N,B)^{1/100}$ in the left-hand side of \eqref{Alex51} is easily bounded by $C\min(N,B)^{-1/100}$. Similarly, the sum over the vectors $v$ with the property that $|v|\leq N\min(N,B)^{1/100}$ and $|p\cdot v|\geq N^2\min(N,B)^{1/10}$ is also bounded by $C\min(N,B)^{-1/100}$. Therefore, letting $\widehat{p}=p/|p|$ and using that $|p|\geq BN$, it remains to prove that
\begin{equation*}
N^{-3}\sum_{|\widehat{p}\cdot v|\leq N\min(N,B)^{-9/10}}\left[1+\frac{\vert v\vert}{N}\right]^{-10}\lesssim \min(N,B)^{-1/100},
\end{equation*}
which is an elementary estimate.
\end{proof}

We will need one more lemma.

\begin{lemma}\label{Al4}
Assume that $\mathcal{O}_\al=(N_{k,\alpha},t_{k,\alpha},x_{k,\alpha})_{k}\in\mathcal{F}_e$, $\al\in\{1,2\}$, are two orthogonal frames, $I\subseteq (-1,1)$ is a fixed open interval, $0\in I$, and $T_1,T_2,R\in[1,\infty)$ are fixed numbers, $R\geq T_1+T_2$. For $k$ large enough let
\begin{equation*}
\mathcal{S}_{k,\alpha}=\{(x,t)\in\T^3\times I:|t-t_{k,\alpha}|< T_\alpha N_{k,\alpha}^{-2},\,|x-x_{k,\alpha}|\leq R N_{k,\alpha}^{-1}\}.
\end{equation*}
Assume that $(\omega_{k,1},\omega_{k,2},f_k,g_k,h_k)_k$ are 5 sequences of functions with the properties
\begin{equation}\label{Al5}
\begin{split}
&|D_x^m\omega_{k,\alpha}|+N_{k,\alpha}^{-2}\mathbf{1}_{\mathcal{S}_{k,\alpha}}\cdot |\partial_tD_x^m\omega_{k,\alpha}|\leq R N_{k,\alpha}^{1/2+|m|}\mathbf{1}_{\mathcal{S}_{k,\alpha}},\quad 0\leq |m|\leq 4,\,\al\in\{1,2\},\\
&\|f_k\|_{X^1(I)}\leq 1,\quad \|g_k\|_{X^1(I)}\leq 1,\quad \|h_k\|_{X^1(I)}\leq 1,
\end{split}
\end{equation}
for any $k$ sufficiently large. Then
\begin{equation*}
\lim_{k\to\infty}\|\omega_{k,1}\omega_{k,2}f_kg_kh_k\|_{N^1(I)}=0.
\end{equation*}
\end{lemma}

\begin{proof}[Proof of Lemma \ref{Al4}] Fix $\eps>0$ small. If $N_{k,1}/N_{k,2}+N_{k_2}/N_{k,1}\leq\eps^{-1000}$ and $k$ is large enough then $\mathcal{S}_{k,1}\cap\mathcal{S}_{k,2}=\emptyset$, thus
\begin{equation*}
\omega_{k,1}\omega_{k,2}f_kg_kh_k\equiv 0.
\end{equation*}

If
\begin{equation}\label{Al5.5}
N_{k,1}/N_{k,2}\geq\eps^{-1000}/2
\end{equation}
we observe that
\begin{equation*}
\omega_{k,1}\omega_{k,2}=\omega_{k,1}\widetilde{\omega_{k,2}}:=\omega_{k,1}\cdot (\omega_{k,2}\mathbf{1}_{(t_{k,1}-T_1N_{k,1}^{-2},t_{k,1}+T_1N_{k,1}^{-2})}(t))
\end{equation*}
and
\begin{equation}\label{Al6}
\|\widetilde{\omega_{k,2}}\|_{X^1(I)}\lesssim_R 1,\qquad \|\widetilde{\omega_{k,2}}\|_{Z(I)}\lesssim_R \eps,\qquad \|P_{>\eps^{-10}N_{k_2}}\widetilde{\omega_{k,2}}\|_{X^1(I)}\lesssim_R \eps,
\end{equation}
where the bound on the $X^1$-norm above and below is computed using \eqref{EstimX1Norm}.
Also, we decompose
\begin{equation}\label{Al7}
\begin{split}
&\omega_{k,1}=P_{\leq \eps^{50} N_{k,1}}\omega_{k,1}+P_{>\eps^{50} N_{k,1}}\omega_{k,1},\\
&\|\omega_{k,1}\|_{X^1(I)}\lesssim _R1, \qquad \|P_{\leq \eps^{50} N_{k,1}}\omega_{k,1}\|_{X^1(I)}\lesssim _R\eps.
\end{split}
\end{equation}
Using Lemma \ref{NLEst2}, \eqref{NLEst3}, and the bounds \eqref{Al6}, \eqref{Al7}, we estimate, assuming \eqref{Al5.5},
\begin{equation*}
\begin{split}
\|\omega_{k,1}\omega_{k,2}f_kg_kh_k\|_{N^1(I)}&\lesssim \|(P_{\leq \eps^{50} N_{k,1}}\omega_{k,1})\widetilde{\omega_{k,2}}f_kg_hh_k\|_{N^1(I)}\\
&+\|(P_{> \eps^{50} N_{k,1}}\omega_{k,1})(P_{>\eps^{-10}N_{k_2}}\widetilde{\omega_{k,2}})f_kg_kh_k\|_{N^1(I)}\\
&+\|(P_{> \eps^{50} N_{k,1}}\omega_{k,1})(P_{\leq\eps^{-10}N_{k_2}}\widetilde{\omega_{k,2}})f_kg_kh_k\|_{N^1(I)}\\
&\lesssim_R\eps^{1/2}.
\end{split}
\end{equation*}
The conclusion of the lemma follows.
\end{proof}

We turn now to the proof of Lemma \ref{AlmostSol}. We will use repeatedly the following description of the nonlinear profiles $U_k^\gamma$. Using Proposition \ref{GEForEP}, Lemma \ref{step1} and Lemma \ref{Extinction}, it follows that for any $\theta>0$ there is $T^0_{\theta,\gamma}=T^0_{\theta,\psi^\gamma}$ sufficiently large such that for all $T_{\theta,\gamma}\geq T^0_{\theta,\gamma}$ there is $R_{\theta,\gamma}$ sufficiently large such that for all $k$ sufficiently large (depending on $R_{\theta,\gamma}$) we can decompose
\begin{equation}\label{DecNP}
\begin{split}
&\mathbf{1}_{(-T_{\theta,\gamma}^{-1},T_{\theta,\gamma}^{-1})}(t)U^\gamma_k=V^{\gamma,\theta}_k+\rho^{\gamma,\theta}_k=\omega^{\gamma,\theta,-\infty}_k+\omega^{\gamma,\theta}_k+\omega^{\gamma,\theta,+\infty}_k+\rho^{\gamma,\theta}_k,\\
&\Vert \omega^{\gamma,\theta,\pm\infty}_k\Vert_{Z^\prime(-T_{\theta,\gamma}^{-1},T_{\theta,\gamma}^{-1})}+\Vert\rho^{\gamma,\theta}_k\Vert_{X^1(-T_{\theta,\gamma}^{-1},T_{\theta,\gamma}^{-1})}\le\theta,\\
&\Vert \omega^{\gamma,\theta,\pm\infty}_k\Vert_{X^1(-T_{\theta,\gamma}^{-1},T_{\theta,\gamma}^{-1})}+\Vert\omega^{\gamma,\theta}_k\Vert_{X^1(-T_{\theta,\gamma}^{-1},T_{\theta,\gamma}^{-1})}\lesssim 1,\\
&\vert D_x^m\omega_k^{\gamma,\theta}\vert+(N_k^\gamma)^{-2}\mathbf{1}_{\mathcal{S}^{\gamma,\theta}_k}\vert \partial_tD_x^m\omega_k^{\gamma,\theta}\vert\le R_{\theta,\gamma} (N^\gamma_k)^{1/2+|m|}\mathbf{1}_{\mathcal{S}^{\gamma,\theta}_k},\qquad 0\leq|m|\leq 6,\\
&\omega_k^{\gamma,\theta,\pm\infty}=\mathbf{1}_{\{\pm(t-t^\gamma_k)\ge T_{\theta,\gamma}(N^\gamma_k)^{-2},\,\vert t\vert\le T_{\theta,\gamma}^{-1}\}}[e^{i(t-t_k^\gamma)\Delta}\pi_{x_k^\gamma}T_{N_k^\gamma}(\phi^{\gamma,\theta,\pm\infty})],\\
&\|\phi^{\gamma,\theta,\pm\infty}\|_{\dot{H}^1(\R^3)}\lesssim 1,\qquad\phi^{\gamma,\theta,\pm\infty}=P_{\le R_{\theta,\gamma}}(\phi^{\gamma,\theta,\pm\infty}),\qquad \|\phi^{\gamma,\theta,\pm\infty}\|_{L^1(\R^3)}\leq R_{\theta,\gamma},
\end{split}
\end{equation}
where
\begin{equation*}
\mathcal{S}^{\gamma,\theta}_k:=\{(x,t)\in\T^3\times (-T_{\theta,\gamma}^{-1},T_{\theta,\gamma}^{-1}):\vert t-t^\gamma_k\vert<T_{\theta,\gamma}(N^\gamma_k)^{-2},\,\vert x-x^\gamma_k\vert\le R_{\theta,\gamma}(N^\gamma_k)^{-1}\}.
\end{equation*}
Indeed, one starts by examining \eqref{ProxyEuclHyp} and \eqref{ScatEuclSol} with $T_{\theta,\gamma}$ sufficiently large; all terms that have small $X^1$ norm are included in $\rho^{\gamma,\theta}_k$. The core $\omega^{\gamma,\theta}_k$ corresponds to the main term in \eqref{ProxyEuclHyp} and the scattering components $\omega^{\gamma,\theta,\pm\infty}_k$ correspond to the main terms in \eqref{ScatEuclSol} (after an additional regularization that produces additional acceptable $X^1$ errors).  

In addition, since $\|W\|_{X^1(-1,1)}\lesssim 1$, for any $\theta>0$ there is $T_\theta>0$ such that
\begin{equation}\label{DecNP2}
\Vert W\Vert_{Z^\prime(-T_{\theta}^{-1},T_{\theta}^{-1})}\le\theta,\quad \Vert W\Vert_{X^1(-T_{\theta}^{-1},T_{\theta}^{-1})}\lesssim 1.
\end{equation} 

\begin{proof}[Proof of  \eqref{Asol1}] For fixed $J$, we have that
\begin{equation*}
F(U^J_{prof,k})-F(W)-\sum_{1\le\alpha\le J}F(U^\alpha_k)
\end{equation*}
can be expressed as a finite linear combination of products of the form
\begin{equation}\label{exp}
W^1_kW^2_kW^3_kW^4_kW^5_k
\end{equation}
for $W_k^i\in\{W,\overline{W},U^\alpha_k,\overline{U}^\alpha_k,1\le\alpha\le J\}$, with at least two terms differing by more than just complex conjugation.  

Assume $\theta>0$ is fixed. We further decompose the profiles $U_k^\alpha$, $1\leq\al\leq J$ according to the first line in \eqref{DecNP} and set
\begin{equation*}
T_{\theta,\alpha}=T_{\theta,\be}:=T_\theta\text{ for any }\al,\be\in\{1,\ldots,J\}.
\end{equation*}
For $k$ large enough, all expressions arising from a product as in \eqref{exp} containing an error term $\rho_k^{\alpha,\theta}$ are $\lesssim\theta$ in the $N^1(I_k)$ norm, in view of Lemma \ref{NLEst2}. Similarly, all expressions containing two scattering components $\omega^{\al,\theta,\pm\infty}_k$ and $\omega^{\be,\theta,\pm\infty}_k$ (or one scattering component $\omega^{\al,\theta,\pm\infty}_k$ and $W$) are also $\lesssim\theta$ in the $N^1(I_k)$ norm, using again Lemma \ref{NLEst2}. All expressions containing two different cores $\omega^{\al,\theta}_k$ and $\omega^{\be,\theta}_k$, $\al\neq\be$, converge to $0$ in the $N^1(I_k)$ norm, in view of Lemma \ref{Al4}. Therefore it remains to prove that
\begin{equation}\label{Al10}
\limsup_{k\to\infty}\|\mathfrak{O}_{4,1}(\omega^{\be,\theta}_k,\omega^{\al,\theta,\pm\infty}_k)\|_{N^1(I_k)}\lesssim\theta^{1/10},
\end{equation}
for any $\al=0,1,\ldots J$, $\be=1,2,\ldots J$, $\al\neq\be$, $\omega^{0,\theta,\pm\infty}:= W\cdot\mathbf{1}_{(-T_\theta^{-1},T_\theta^{-1})}(t)$.

Let $N_k^0:=1$. The limit \eqref{Al10} is an easy consequence of Lemma \ref{HFLF} if
\begin{equation*}
\lim_{k\to\infty}N_k^\al/N_k^\be=\infty.
\end{equation*}
If
\begin{equation*}
\lim_{k\to\infty}N_k^\be/N_k^\al=\infty.
\end{equation*}
the limit \eqref{Al10} follows from \eqref{NLEst3}, after decomposing as in \eqref{Al7} and using the smallness in $Z'(I_k)$ of $\omega^{\al,\theta,\pm\infty}_k$. If\footnote{Recall our assumptions on the frames $\mathcal{O}^\al$ described at the beginning of {\bf{Case III}} in the previous section.}
\begin{equation*}
N_k^\al=N_k^\be \text{ and }t_k^\al=t_k^\be \text{ as }k\to\infty  
\end{equation*}
then $\omega^{\be,\theta}_k\omega^{\al,\theta,\pm\infty}_k=0$. Finally, assume that
\begin{equation*}
N_k^\al=N_k^\be \text{ and }\lim_{k\to\infty}N_k^\al|t_k^\al-t_k^\be|^{1/2}=\infty.
\end{equation*}
Fix $\widetilde{\phi}^{\al,\theta,\pm\infty}\in C^\infty_0(\R^3)$ such that $\|\widetilde{\phi}^{\al,\theta,\pm\infty}-\phi^{\al,\theta,\pm\infty}\|_{\dot{H}^1}\leq\theta$ and define
\begin{equation*}
\widetilde{\omega}_k^{\al,\theta,\pm\infty}=\mathbf{1}_{\{\pm(t-t^\al_k)\ge T_{\theta,\al}(N^\al_k)^{-2},\,\vert t\vert\le T_{\theta,\al}^{-1}\}}[e^{i(t-t_k^\al)\Delta}\pi_{x_k^\al}T_{N_k^\al}(\widetilde{\phi}^{\al,\theta,\pm\infty})].
\end{equation*}
Then, using Lemma \ref{NLEst2}, for all $k$ sufficiently large
\begin{equation*}
\|\mathfrak{O}_{4,1}(\omega^{\be,\theta}_k,(\omega^{\al,\theta,\pm\infty}_k-\widetilde{\omega}_k^{\al,\theta,\pm\infty}))\|_{N^1(I_k)}\lesssim\theta
\end{equation*}
Moreover, using \eqref{ELiL1} with $T=N_k|t_k^\al-t_k^\be|^{1/2}$ and $p=\infty$, it follows easily that
\begin{equation*}
\limsup_{k\to\infty}\|\mathfrak{O}_{4,1}(\omega^{\be,\theta}_k,\widetilde{\omega}^{\al,\theta,\pm\infty}_k)\|_{L^1(I_k,H^1)}=0.
\end{equation*}
This completes the proof of \eqref{Al10}.
\end{proof}

\begin{proof}[Proof of \eqref{Asol2}] We compute that, for fixed $J$,
\begin{equation*}
\|F(U^J_{prof,k}+e^{it\Delta}R^J_k)-F(U^J_{prof,k})\|_{N(I_k)}\lesssim\sum_{p=0}^4\|\mathfrak{O}_{p,5-p}(U^J_{prof,k},e^{it\Delta}R^J_k)\|_{N(I_k)},
\end{equation*}
where $\mathfrak{O}_{p,q}(a,b)$ stands for a $p+q$-linear expression with $p$ factors consisting of either $a$ or $\overline{a}$ and $q$ factors consisting of either $b$ or $\overline{b}$. Using Lemma \eqref{NLEst2} and the fact that $U^J_{prof,k}$ is uniformly bounded in $X^1$, we can control the terms corresponding to $p\le 3$ as follows
\begin{equation*}
\Vert \mathfrak{O}_{p,5-p}(U^J_{prof,k},e^{it\Delta}R^J_k)\Vert_{N(I)}\lesssim \Vert e^{it\Delta}R^J_k\Vert_{X^1(I)}^{4-p}\Vert e^{it\Delta}R^J_k\Vert_{Z^\prime(I)}\Vert U^J_{prof,k}\Vert_{X^1(I)}^p\lesssim \Vert e^{it\Delta}R^J_k\Vert_{Z^\prime(I)}.
\end{equation*}
In view of \eqref{SmallnessRterm}, this contribution is acceptable. 

Now, we only need to treat the contribution of $p=4$. Assume $\eps>0$ is fixed. As in the proof of \eqref{bi1} (using also Lemma \ref{EquivFrames} (ii)), there is $A=A(\eps)$ sufficiently large such that for all $J\geq A$ and all $k\geq k_0(J)$
\begin{equation*}
\|U^J_{prof,k}-U^A_{prof,k}\|_{X^1(-1,1)}\leq\eps.
\end{equation*}
In view of Lemma \ref{NLEst2}, it remains to prove that
\begin{equation*}
\limsup_{J\to\infty}\limsup_{k\to\infty}\Vert \mathfrak{O}_{4,1}(U^A_{prof,k},e^{it\Delta}R^J_k)\Vert_{N(I_k)}\lesssim\eps.
\end{equation*}
Using the definition of $U^A_{prof,k}$, it suffices to prove that for any $\al_1,\al_2,\al_3,\al_4\in\{0,1,\ldots,A\}$
\begin{equation}\label{Al30}
\limsup_{J\to\infty}\limsup_{k\to\infty}\Vert \mathfrak{O}_{1,1,1,1,1}(U_k^{\al_1},U_k^{\al_2},U_k^{\al_3},U_k^{\al_4},e^{it\Delta}R^J_k)\Vert_{N(I_k)}\lesssim\eps A^{-4},
\end{equation}
where $U_k^0:=W$.

Fix $\theta=(\eps A^{-4})^{10}$ and apply the decomposition in the first line of \eqref{DecNP} to all nonlinear profiles $U_k^{\al}$, $\al=0,1,\ldots,A$. We may also assume that
\begin{equation*}
T_{\theta,\al}=T_\theta\text{ and }R_{\theta,\alpha}=R_\theta\text{ for any }\al=1,\ldots,A,
\end{equation*}
and that all the bounds in \eqref{DecNP} and \eqref{DecNP2} hold. Using these decompositions we examine now the terms in the expression
\begin{equation*}
\mathfrak{O}_{1,1,1,1,1}(U_k^{\al_1},U_k^{\al_2},U_k^{\al_3},U_k^{\al_4},e^{it\Delta}R^J_k).
\end{equation*}
Recalling that
\begin{equation}\label{Al40}
\limsup_{J\to+\infty}\limsup_{k\to+\infty}\Vert e^{it\Delta}R^J_k\Vert_{Z(I_k)}=0, 
\end{equation}
see \eqref{SmallnessRterm}, and using Lemma \ref{NLEst2}, for \eqref{Al30} it suffices to prove that
\begin{equation*}
\limsup_{J\to\infty}\limsup_{k\to\infty}\Vert \mathfrak{O}_{1,1,1,1,1}(\omega_k^{\al_1,\theta},\omega_k^{\al_2,\theta},\omega_k^{\al_3,\theta},\omega_k^{\al_4,\theta},e^{it\Delta}R^J_k)\Vert_{N(I_k)}\lesssim\eps A^{-4},
\end{equation*}
for any $\al_1,\al_2,\al_3,\al_4\in\{1,\ldots,A\}$. Using Lemma \ref{Al4}, we only need to consider the case $\al_1=\al_2=\al_3=\al_4$, i.e. it remains to prove that for any $\al\in\{1,\ldots,A\}$
\begin{equation}\label{Al31}
\limsup_{J\to\infty}\limsup_{k\to\infty}\Vert \mathfrak{O}_{4,1}(\omega_k^{\al,\theta},e^{it\Delta}R^J_k)\Vert_{N(I_k)}\lesssim\eps A^{-4}.
\end{equation}
 
We apply Lemma \ref{HFLF} with $B$ sufficiently large (depending on $R_\theta$), thus, for any $J\geq A$,
\begin{equation}\label{Al32}
\limsup_{k\to\infty}\Vert \mathfrak{O}_{4,1}(\omega_k^{\al,\theta},P_{>BN_k^\al}e^{it\Delta}R^J_k)\Vert_{N(I_k)}\lesssim\eps A^{-4}.
\end{equation}
We may assume that $B$ is sufficiently large such that, for $k$ large
\begin{equation*}
\|P_{\leq B^{-1}N_k^\al}\omega_k^{\al,\theta}\|_{X^1(I_k)}\leq \eps A^{-4}.
\end{equation*}
Using Lemma \ref{NLEst2} and the bounds \eqref{Al40} and \eqref{Al32}, for \eqref{Al31} it remains to prove that
\begin{equation*}
\limsup_{J\to\infty}\limsup_{k\to\infty}\Vert \mathfrak{O}_{4,1}(P_{>B^{-1}N_k^\al}\omega_k^{\al,\theta},P_{\leq BN_k^\al}e^{it\Delta}R^J_k)\Vert_{N(I_k)}=0.
\end{equation*}
This follows from \eqref{NLEst3} and \eqref{Al40}, which completes the proof.
\end{proof}

\end{document}